\documentclass[11pt]{amsart}
\usepackage{amsmath,amssymb,amsthm,mathtools,mathrsfs}
\usepackage[a4paper,margin=2.7cm]{geometry}
\usepackage[colorlinks=true,linkcolor=blue,citecolor=blue,urlcolor=blue]{hyperref}

\newtheorem{theorem}{Theorem}[section]

\newtheorem{lemma}[theorem]{Lemma}
\newtheorem{corollary}[theorem]{Corollary}
\theoremstyle{definition}

\theoremstyle{remark}
\newtheorem{remark}[theorem]{Remark}
\newtheorem*{ack}{\bf Acknowledgements}
\theoremstyle{plain}
\newtheorem{thmletter}{Theorem}

\numberwithin{equation}{section}
\allowdisplaybreaks[1]

\title[The Yau Geometric Conjecture]{The Yau Geometric Conjecture for Isolated Hypersurface Singularities}

\author{Siyuan Fu}
\address{Weiyang College, Tsinghua University}
\email{fusy24@mails.tsinghua.edu.cn}

\author{Stephen S.-T. Yau}
\address{Department of Mathematical Sciences, Tsinghua University, Beijing, China}
\email{yau@uic.edu}

\author{Huaiqing Zuo}
\address{Department of Mathematical Sciences,
	Tsinghua University,
	Beijing, 100084, P. R. China.}
\email{hqzuo@mail.tsinghua.edu.cn}

\date{}

\begin{document}

\begin{abstract}

This paper presents a proof of the long-standing Yau Geometric Conjecture.


\end{abstract}

\maketitle

\section{Introduction}

The Milnor number, geometric genus, and multiplicity are fundamental
invariants of an isolated hypersurface singularity. Inequalities among
these invariants give quantitative information about the singularity;
their equality cases can also characterize its analytic structure. For
weighted homogeneous singularities, the Milnor--Orlik product formula and
the geometric genus lattice-point formula connect this problem with
counting positive integral points in a simplex
\cite{MO70,MT80,YZ15,LRYZ18}.

Let $n\ge3$, and let
$f\colon(\mathbb C^n,0)\to(\mathbb C,0)$ be a weighted homogeneous
polynomial with an isolated singularity at the origin. Write
$V=\{f=0\}$, and denote its Milnor number, geometric genus, and
multiplicity by $\mu$, $p_g$, and $v$, respectively. Thus $n$ is the
dimension of the ambient space, and $\dim V=n-1$. For a real number $t$,
put
\begin{equation}\label{eq:pn-definition}
 (t)_n=\prod_{j=0}^{n-1}(t-j),\qquad
 p_n(t)=(t-1)^n-(t)_n.
\end{equation}
For an integer $v\ge2$, this gives
\[
 p_n(v)=(v-1)^n-n!\dbinom vn,
\]
where $\dbinom vn=0$ when $v<n$.

The \textbf{Yau Geometric Conjecture}, announced in 1995
\cite{YZ06}, predicts the sharp inequality
\begin{equation}\label{eq:ygc}
 \mu-p_n(v)\ge n!p_g.
\end{equation}
Equality holds if and only if $f$ is right equivalent to a
homogeneous polynomial. Here right equivalence means that a local
biholomorphic change of coordinates transforms one defining function
into the other. The homogeneous case explains the correction term in
\eqref{eq:ygc}: for a homogeneous isolated singularity of degree $v$,
\[
 \mu=(v-1)^n,\qquad p_g=\dbinom vn.
\]

Several low-dimensional cases were established before the general
case. In ambient dimension three, Xu and Yau proved the required
inequality and characterized its equality case
\cite[Theorem~A]{XY93}. The cases of ambient dimensions four and five
with positive geometric genus were treated by Lin and Yau
\cite[Main Theorem]{LY04} and by Chen, Lin, Yau, and Zuo
\cite[Theorem~1.4]{CLYZ11}, respectively. Lin, Raghuvanshi, Yau, and
Zuo obtained the corresponding six-dimensional lattice-point estimate
and its geometric application \cite[Theorem~1.10 and Section~4]{LRYZ18}.
Yau and Zuo proved, in every dimension, the multiplicity bound
$\mu\ge(v-1)^n$, with equality exactly for singularities right
equivalent to homogeneous ones \cite[Theorem~A]{YZ15}. Their result
also establishes the Yau Geometric Conjecture when $p_g=0$
\cite[Theorem~C]{YZ15}. Taken together, these results settle the
conjecture for $3\le n\le6$.

The purpose of this paper is to prove the remaining cases and hence
establish the following theorem.

\begin{thmletter}\label{thm:main-introduction}
Let $n\ge3$, and let
$f\colon(\mathbb C^n,0)\to(\mathbb C,0)$ be a weighted homogeneous
polynomial defining an isolated hypersurface singularity at the origin.
Then
\[
 n!p_g\le\mu-(v-1)^n+n!\dbinom vn.
\]
Equality holds if and only if $f$ is right equivalent to a
homogeneous polynomial.
\end{thmletter}

We use inverse normalized weights
$a_1\ge\cdots\ge a_n>1$, so that each monomial $z^\alpha$ occurring
in $f$ satisfies $\sum_i\alpha_i/a_i=1$. In this notation,
\[
 p_g=P_n(a):=
 \#\left\{x\in\mathbb Z_{>0}^n:
                 \sum_{i=1}^n\cfrac{x_i}{a_i}\le1\right\},
 \qquad
 \mu=\prod_{i=1}^n(a_i-1).
\]
Our first step is an estimate that applies to arbitrary real
intercepts, without assuming that they arise from a singularity.

\begin{thmletter}\label{thm:lattice-introduction}
Let $n\ge3$ and let $r\ge2$ be an integer. If
$a_1\ge\cdots\ge a_n\ge r$ are real numbers, then
\[
 n!P_n(a)\le\prod_{i=1}^n(a_i-1)-p_n(r).
\]
Equality holds if and only if $a_1=\cdots=a_n=r$.
\end{thmletter}

\begin{remark}

Under the additional assumption $P_n(a)>0$, the main theorem of
\cite{LYZ26} gives
\[
n!P_n(a)\le\prod_{i=1}^n(a_i-1)-p_n(a_n).
\]
When $a_n=r$ is an integer, this is precisely the inequality in Theorem~\ref{thm:lattice-introduction}. 
For general $a_n\ge r$, the relative strength of the two estimates is determined by the
comparison between $p_n(a_n)$ and $p_n(r)$. 
However, there is no uniform ordering between $p_n(a_n)$ and $p_n(r)$
as $n$, $r$, and $a_n$ vary in the above range. For example, set \(r=2,\, a_n=5/2\) and we have
\[
p_4\!\left(\frac52\right)=6>p_4(2)=1,\qquad p_9\!\left(\frac52\right)
=-\frac{68121}{256}<p_9(2)=1.
\]

\end{remark}

This estimate follows by induction from the three-dimensional sharp
bound of Xu and Yau \cite{XY92}, with the estimate of Yau and Zhang
\cite{YZ06} handling the slices for which the slice parameter equals one. It proves
Theorem~\ref{thm:main-introduction} whenever the smallest inverse weight
is an integer.

For the remaining case, set $v=\lceil a_n\rceil$ and $u=v-1$.
The isolated singularity condition forces each inverse weight
$b\in(u,v)$ to have a partner $c=b/(b-u)>v$, with distinct partner
variables for the variables whose inverse weights are below $v$. This pairing permits
a lower estimate for the Milnor product. Replacing each small
intercept by $v$ gives an upper estimate for the lattice-point count,
which we compute by fixing the corresponding coordinates simultaneously.
After expansion, all required coefficient comparisons follow from
one estimate for the highest-degree coefficient. A midpoint
inequality on unit cubes supplies that estimate for $n\ge7$.

The paper is organized as follows. Section~\ref{sec:preliminaries}
fixes the notation, proves the
coordinate-axis relation, and recalls the external estimates.
Section~\ref{sec:integer-reference} proves
Theorem~\ref{thm:lattice-introduction} (see Theorem \ref{thm:integer-reference}), including its equality
condition. Section~\ref{sec:ygc-proof} proves the geometric inequality
for $n\ge7$ and combines it with the known low-dimensional results.

\begin{ack}
AI tools (ChatGPT and DeepSeek) were used during
the preparation of the paper as technical assistants to help with checking calculations and
arguments, developing details of proofs using standard techniques, and improving the exposition. All mathematical statements, proofs, and conclusions were verified by the authors.  H. Zuo was supported by NSFC Grant 12671056 and BJNSF Grant 1252009. Yau was supported by the Tsinghua University Education Foundation.

\end{ack}

\section{Preliminaries}\label{sec:preliminaries}

\subsection{The invariants and right equivalence}

We use the conventions on singularities and critical points in
\cite{GGL07}. Let $\mathcal O_n=\mathbb C\{z_1,\cdots,z_n\}$ and let
$\mathfrak m=(z_1,\cdots,z_n)$ be its maximal ideal. For a holomorphic
function germ \(f\) with an isolated singularity at the origin, its Milnor
number and multiplicity are
\[
 \mu(f)=\dim_{\mathbb C}
 \cfrac{\mathcal O_n}{(\partial f/\partial z_1,\cdots,
                       \partial f/\partial z_n)},
 \qquad
 v(f)=\operatorname{ord}_0f.
\]

Let $(V,0)$ be an isolated hypersurface singularity of complex
dimension $s\ge2$, and let
\[
 \pi\colon(M,E)\longrightarrow(V,0),\qquad E=\pi^{-1}(0),
\]
be a resolution of the singularity. Its geometric genus is
\[
 p_g(V,0)=\dim_{\mathbb C}H^{s-1}(M,\mathcal O_M),
\]
where $\mathcal O_M$ is the sheaf of holomorphic functions on $M$.
This number is independent of the chosen resolution. Here the resolution is taken over a sufficiently
small local representative of the germ, and the cohomology is computed
on the corresponding $M$. If $V=\{f=0\}\subset\mathbb C^n$ and
$n\ge3$, then $s=n-1$, so
\[
 p_g(V,0)=\dim_{\mathbb C}H^{n-2}(M,\mathcal O_M).
\]
For a surface singularity in $\mathbb C^3$, this is
$\dim_{\mathbb C}H^1(M,\mathcal O_M)$.

We use the standard fact that the Milnor number $\mu$,
the geometric genus $p_g$, and the multiplicity $v$
are invariant under right equivalence.

\subsection{Weights and lattice points}

Suppose that
\[
 f(z)=\sum_{\alpha\in\mathbb Z_{\ge0}^n}c_\alpha z^\alpha
\]
is weighted homogeneous of type $(\boldsymbol w;d)$ with an isolated singularity at the origin, in the notation
of \cite{GGL07}. Thus $w_i$ and $d$ are positive integers and
\[
 d=\sum_{i=1}^nw_i\alpha_i
 \qquad\text{whenever }c_\alpha\ne0.
\]
By the Euler relation $d \cdot f=\sum_iw_iz_i\partial f/\partial z_i$, we get \(\operatorname{Sing}V_f=\operatorname{Crit}f\), where \(V_f=\{f=0\}\).

We use the inverse normalized weights
\[
 a_i=\cfrac d{w_i},\qquad q_i=\cfrac1{a_i}=\cfrac{w_i}{d}.
\]
Then weighted homogeneity takes the form
\begin{equation}\label{eq:weighted-homogeneity}
 \sum_{i=1}^n\cfrac{\alpha_i}{a_i}
   =\sum_{i=1}^nq_i\alpha_i=1
 \qquad\text{for every }\alpha\in\mathbb Z_{\ge0}^n
              \text{ with }c_\alpha\ne0.
\end{equation}
After reordering the variables, we assume
\begin{equation}\label{eq:ordered-weights}
 a_1\ge a_2\ge\cdots\ge a_n>1,
 \qquad 0<q_1\le q_2\le\cdots\le q_n<1.
\end{equation}
The $a_i$ are the weights used in \cite{LRYZ18}; geometrically,
they are the intercepts of the simplex defining the lattice-point
count below. This convention differs from the integer weights
$w_i$ in \cite{GGL07}.

For a vector $a=(a_1,\cdots,a_n)$ of positive real numbers, write
\[
 \Delta_n(a)=\left\{x=(x_1,\cdots,x_n)\in\mathbb Z_{>0}^n:
                         \sum_{i=1}^n\cfrac{x_i}{a_i}\le1\right\},
 \qquad P_n(a)=\#\Delta_n(a).
\]
The geometric genus lattice-point formula, attributed to
Merle and Teissier \cite{MT80}, gives
\begin{equation}\label{eq:pg-lattice}
 p_g=P_n(a).
\end{equation}
We use the weighted homogeneous form of this formula recalled in
\cite[p.~218]{YZ15} and \cite[p.~602]{LRYZ18}.
The Milnor--Orlik formula \cite[Theorem~1]{MO70} gives
\begin{equation}\label{eq:milnor-product}
 \mu=\prod_{i=1}^n(a_i-1).
\end{equation}
Throughout the paper, polynomial exponents lie in
$\mathbb Z_{\ge0}^n$, whereas the lattice points counted by $P_n$
have strictly positive coordinates.

Then, we recall a lemma of Saito concerning the monomials
of an isolated singularity.

\begin{lemma}[Saito, {\cite[Lemma~1.5]{Sai71}}]
\label{lem:saito}
Let $f\in\mathbb C[[x_1,\ldots,x_n]]$ have an isolated
singularity at the origin. Then, for every nonempty subset
$J\subset\{1,\ldots,n\}$, at least one of the following holds:
\begin{enumerate}
 \item The series $f$ contains a monomial with nonzero
 coefficient involving only the variables $x_j$ with $j\in J$.

 \item There exists a subset
 \[
  K\subset\{1,\cdots,n\}\setminus J,
  \qquad |K|=|J|,
 \]
 such that, for every $k\in K$, the series $f$ contains
 a monomial $x_km_k$ with nonzero coefficient, where $m_k$
 is a monomial involving only the variables $x_j$ with $j\in J$.
\end{enumerate}
\end{lemma}

\begin{corollary}\label{cor:axis}
We have $v=\lceil a_n\rceil$. Moreover, at least one of the following
types of monomials occurs in $f$ with nonzero coefficient:
\begin{enumerate}
 \item $z_n^v$, in which case $a_n=v$;
 \item $z_n^{v-1}z_j$ for some $j\ne n$, in which case
 \begin{equation}\label{eq:axis-relation}
  (v-1)q_n+q_j=1.
 \end{equation}
\end{enumerate}
In particular, if $a_n$ is not an integer, then for such an index $j$,
\[
 a_n-\lfloor a_n\rfloor=\cfrac{a_n}{a_j}.
\]
\end{corollary}
\begin{proof}
Write $f(z)=\sum_\alpha c_\alpha z^\alpha$.
For every $\alpha\in\mathbb Z_{\ge0}^n$ with $c_\alpha\ne0$,
weighted homogeneity gives
\[
 1=\sum_{i=1}^nq_i\alpha_i
   \le q_n\sum_{i=1}^n\alpha_i.
\]
Hence
\[
 v=\operatorname{ord}_0f
   =\min_{c_\alpha\ne0}\sum_{i=1}^n\alpha_i
   \ge\lceil a_n\rceil.
\]

Applying Lemma~\ref{lem:saito} with $J=\{n\}$, we find that
at least one of the following holds:
\begin{enumerate}
 \item $f$ contains the monomial $z_n^s$ with nonzero coefficient;
 \item $f$ contains the monomial $z_jz_n^r$ with nonzero
 coefficient for some $j\ne n$.
\end{enumerate}
Here $s\ge2$ and $r\ge1$, since $f$ has no constant or linear terms.

In the first case, weighted homogeneity gives
\[
 sq_n=1,
 \qquad\text{so}\qquad a_n=s\in\mathbb Z.
\]
Since $z_n^s$ occurs in $f$, we have
\[
 v\le \deg z_n^s= s=a_n=\lceil a_n\rceil.
\]
Together with the opposite inequality proved above, this gives
$v=\lceil a_n\rceil=s$.

In the second case, weighted homogeneity gives
\[
 rq_n+q_j=1.
\]
Since $0<q_j\le q_n$, we obtain
\[
 r<a_n=r+\cfrac{q_j}{q_n}\le r+1.
\]
Thus $\lceil a_n\rceil=r+1$.
Since $z_jz_n^r$ occurs in $f$, we also have
\[
 v\le \deg (z_jz_n^r)=  r+1=\lceil a_n\rceil.
\]
Consequently, $v=\lceil a_n\rceil$ and $r=v-1$, so
\[
 (v-1)q_n+q_j=1.
\]

If $a_n$ is not an integer, only the second case can occur.
Multiplying \eqref{eq:axis-relation} by $a_n$ gives
\[
 v-1+\cfrac{a_n}{a_j}=a_n.
\]
Since $v-1=\lfloor a_n\rfloor$, it follows that
\[
 a_n-\lfloor a_n\rfloor=\cfrac{a_n}{a_j}.
\]
\end{proof}

\subsection{External estimates}

We recall the three estimates used in the proofs.

\begin{theorem}[Yau--Zhang, {\cite[Theorem~1.1]{YZ06}}]
\label{thm:rough-gly}
Let $n\ge3$ and let $a_1\ge\cdots\ge a_n\ge1$ be real numbers.
Then
\[
 n!P_n(a)\le\prod_{i=1}^n(a_i-1).
\]
Equality holds if and only if $a_n=1$. In particular, the inequality
is strict when all the intercepts are greater than one.
\end{theorem}

\begin{theorem}[Xu--Yau, {\cite[p.~200]{XY92}}]
\label{thm:xu-yau}
For real numbers $a\ge b\ge c\ge2$,
\[
 6P_3(a,b,c)\le(a-1)(b-1)(c-1)-c+1.
\]
Equality holds if and only if $a=b=c$ is an integer. 
\end{theorem}

\begin{theorem}[Yau--Zuo, {\cite[Theorem~A]{YZ15}}]
\label{thm:yau-zuo}
Let $f\colon(\mathbb C^n,0)\to(\mathbb C,0)$ be a weighted
homogeneous polynomial with an isolated singularity at the origin.
If $\mu$ and $v$ denote its Milnor number and multiplicity, then
\[
 \mu\ge(v-1)^n.
\]
Equality holds if and only if $f$ is right equivalent to a
homogeneous polynomial.
\end{theorem}

\section{Lattice-point estimates with an integer lower bound}\label{sec:integer-reference}

In every coefficient expansion below, empty products are taken to be \(1\).
This convention applies in
\eqref{eq:general-product-expansion},
\eqref{eq:slice-product-expansion},
\eqref{eq:paired-milnor-product},
\eqref{eq:delta-product},
\eqref{eq:pstar-product-expansion},
\eqref{eq:A-definition}, and
\eqref{eq:T-definition},
and whenever the same factors occur subsequently. 

\begin{lemma}\label{lem:power-sum}
Let $h\ge2$ be an integer. For every real number $t>1$,
\[
 (h+1)\sum_{k\ge1}\bigl((t-k)_+\bigr)^h<(t-1)t^h,
 \qquad y_+:=\max\{y,0\}.
\]
\end{lemma}

\begin{proof}
Put
\[
 S_h(t):=\sum_{k\ge1}\bigl((t-k)_+\bigr)^h
        =\sum_{k\ge1}\max\{t-k,0\}^h.
\]
Only finitely many summands are nonzero. If $1<t\le2$, then
\[
 S_h(t)=(t-1)^h,
\]
so it suffices to prove
\[
 (h+1)(t-1)^{h-1}<t^h.
\]
This follows by differentiation. Indeed, set
\[
 g(t):=\log(h+1)+(h-1)\log(t-1)-h\log t.
\]
Then
\[
 g'(t)=\cfrac{h-1}{t-1}-\cfrac ht
      =\cfrac{h-t}{t(t-1)}.
\]
Thus $g'(t)\ge0$ for $1<t\le2$, while
\[
 g(2)=\log\!\left(\cfrac{h+1}{2^h}\right)<0
 \qquad(h\ge2).
\]
This proves the required inequality on $(1,2]$.

For every $t>1$, we have
\[
 S_h(t+1)=S_h(t)+t^h.
\]
Assume that
\[
 (h+1)S_h(t)<(t-1)t^h.
\]
We wish to show that
\[
 (h+1)S_h(t+1)<t(t+1)^h.
\]
The binomial expansion gives
\[
 (t+1)^h=\sum_{i=0}^h\dbinom hi t^i
 \ge t^h+ht^{h-1}.
\]
Consequently,
\[
\begin{aligned}
 (h+1)S_h(t+1)
   &=(h+1)\bigl(S_h(t)+t^h\bigr)\\
   &\le(t-1)t^h+(h+1)t^h\\
   &=(t+h)t^h<t(t+1)^h.
\end{aligned}
\]
The final inequality is strict because $h\ge2$, so the omitted terms
in the binomial expansion are positive. Induction over successive unit
intervals proves the lemma.
\end{proof}

\begin{lemma}\label{lem:pn-monotone}
For every integer $n\ge3$, the sequence $p_n(m)$ is strictly increasing
for integers $m\ge2$, where
\[
\begin{aligned}
 p_n(m)&=(m-1)^n-m(m-1)\cdots(m-n+1)\\
       &=(m-1)^n-n!\dbinom mn.
\end{aligned}
\]
\end{lemma}

\begin{proof}
We use the notation
\[
 (m)_n:=n!\dbinom mn=m(m-1)\cdots(m-n+1),
\]
with the convention that $\dbinom mn=0$ when $m<n$. We have
\[
\begin{aligned}
 p_n(m+1)-p_n(m)
   &=m^n-(m-1)^n\\
   &\quad-\bigl((m+1)-(m-n+1)\bigr)(m)_{n-1}\\
   &=m^n-(m-1)^n-n(m)_{n-1}\\
   &=m^n-(m-1)^n-n!\dbinom m{n-1}.
\end{aligned}
\]
If $m\le n-2$, the binomial coefficient vanishes by convention, and
therefore
\[
 p_n(m+1)-p_n(m)=m^n-(m-1)^n>0.
\]
If $m\ge n-1$, the function $x\mapsto nx^{n-1}$ is strictly convex on
$[m-1,m]$, since $n\ge3$. Hence
\[
\begin{aligned}
 m^n-(m-1)^n
   &=\int_{m-1}^m nx^{n-1}\,dx\\
   &>n\left(m-\cfrac12\right)^{n-1}.
\end{aligned}
\]
On the other hand, the arithmetic--geometric mean inequality yields
\[
\begin{aligned}
 (m)_{n-1}
   &=m(m-1)\cdots(m-n+2)\\
   &\le\left(m-\cfrac{n-2}{2}\right)^{n-1}\\
   &\le\left(m-\cfrac12\right)^{n-1}.
\end{aligned}
\]
Combining these inequalities gives $p_n(m+1)-p_n(m)>0$.
\end{proof}

\begin{theorem}[Theorem~\ref{thm:lattice-introduction}, restated]\label{thm:integer-reference}
Let $n\ge3$ and $r\ge2$ be integers, and let the real intercepts
$a=(a_1,\cdots,a_n)\in\mathbb R^n$ satisfy $a_i\ge r$ for all $i$.
Write
\[
 P_n(a):=\#\left\{
 x=(x_1,\cdots,x_n)\in\mathbb Z_{>0}^n:
 \sum_{i=1}^n\cfrac{x_i}{a_i}\le1
 \right\}.
\]
Then
\[
 n!P_n(a)\le\prod_{i=1}^n(a_i-1)-p_n(r),
\]
with equality if and only if $a_1=\cdots=a_n=r$.
\end{theorem}

\begin{proof}
We argue by induction on $n$. For $n=3$, put $t:=\min_i a_i\ge r$.
Since $p_3(r)=r-1$, Theorem~\ref{thm:xu-yau} gives
\[
\begin{aligned}
 6P_3(a_1,a_2,a_3)
   &\le(a_1-1)(a_2-1)(a_3-1)-(t-1)\\
   &\le(a_1-1)(a_2-1)(a_3-1)-(r-1).
\end{aligned}
\]
Equality holds in the first inequality if and only if
$a_1=a_2=a_3$ is an integer, and in the second if and only if $t=r$.
This proves the base case, including the equality condition.

Now let $n\ge4$ and put $d:=n-1\ge3$. Assume the theorem in dimension
$d$. Thus, for every integer $r\ge2$ and all $b_i\ge r$,
\[
 d!P_d(b_1,\cdots,b_d)
 \le\prod_{i=1}^d(b_i-1)-p_d(r),
\]
with equality if and only if $b_1=\cdots=b_d=r$.
Moreover, Theorem~\ref{thm:rough-gly} gives
\[
 d!P_d(b_1,\cdots,b_d)\le\prod_{i=1}^d(b_i-1)
 \qquad(b_i\ge1),
\]
with equality if and only if some $b_i=1$.

Consider $a=(a_1,\cdots,a_n)$ with $a_i\ge r\ge2$.
After reordering, we may assume
\[
 a_1\ge a_2\ge\cdots\ge a_n=:t\ge r\ge2.
\]
Set
\[
 R:=\lfloor t\rfloor,\qquad \beta:=t-R,
 \qquad a_i=t+\delta_i\quad(1\le i\le d),
\]
where $\delta_i\ge0$. For the slices below, also put
$a':=(a_1,\cdots,a_d)$. Let
\[
 \Delta_n(a):=\left\{
 x=(x_1,\cdots,x_n)\in\mathbb Z_{>0}^n:
 \sum_{i=1}^n\cfrac{x_i}{a_i}\le1
 \right\}.
\]
For $x\in\Delta_n(a)$, the inequality
\[
 \sum_{i=1}^d\cfrac{x_i}{a_i}+\cfrac{x_n}{t}\le1
\]
implies $x_n<t$, hence $x_n\le R$. We may therefore write
$x_n=R-l$, with $l=0,1,\cdots,R-1$, and obtain the disjoint
decomposition
\[
\begin{aligned}
 \Delta_n(a)
 &=\bigsqcup_{l=0}^{R-1}
 \left\{(x_1,\cdots,x_d,R-l)\in\mathbb Z_{>0}^n:
 \sum_{i=1}^d\cfrac{x_i}{a_i}+\cfrac{R-l}{t}\le1\right\}\\
 &=\bigsqcup_{l=0}^{R-1}
 \left\{(x_1,\cdots,x_d,R-l)\in\mathbb Z_{>0}^n:
 \sum_{i=1}^d\cfrac{x_i}{a_i}\le\cfrac{\beta+l}{t}\right\}.
\end{aligned}
\]
Denote the slice in the second line by
\[
 \Delta_n^{(l)}(a):=
 \left\{(x_1,\cdots,x_d,R-l)\in\mathbb Z_{>0}^n:
 \sum_{i=1}^d\cfrac{x_i}{a_i}\le\cfrac{\beta+l}{t}\right\}.
\]

For $l\ne0$, the defining condition is equivalent to
\[
 \sum_{i=1}^d\cfrac{t}{\beta+l}\cfrac{x_i}{a_i}\le1.
\]
Projection onto the first $d$ coordinates therefore gives
\[
 \#\Delta_n^{(l)}(a)
 =\#\Delta_d\!\left(\cfrac{\beta+l}{t}a'\right)
 =P_d\!\left(\cfrac{\beta+l}{t}a'\right).
\]
Since $a_i\ge a_n=t$, we have
\[
 \cfrac{\beta+l}{t}a_i\ge\beta+l\ge l.
\]
For $l>1$, the induction hypothesis consequently yields
\[
\begin{aligned}
 d!\,\#\Delta_n^{(l)}(a)
   &=d!P_d\!\left(\cfrac{\beta+l}{t}a'\right)\\
   &\le\prod_{i=1}^d
        \left(\cfrac{\beta+l}{t}a_i-1\right)-p_d(l).
\end{aligned}
\]
For $l=1$, Theorem~\ref{thm:rough-gly} gives
\[
\begin{aligned}
 d!\,\#\Delta_n^{(1)}(a)
   &=d!P_d\!\left(\cfrac{\beta+1}{t}a'\right)\\
   &\le\prod_{i=1}^d
        \left(\cfrac{\beta+1}{t}a_i-1\right).
\end{aligned}
\]
This agrees with the preceding expression because $p_d(1)=0$.

For $l=0$, we have
\[
 \Delta_n^{(0)}(a)=
 \left\{(x_1,\cdots,x_d,R)\in\mathbb Z_{>0}^n:
 \sum_{i=1}^d\cfrac{x_i}{a_i}\le\cfrac\beta t\right\}.
\]
If $\beta=0$, this slice is empty, so $\#\Delta_n^{(0)}(a)=0$.
If $\beta>0$ but the slice is empty, its cardinality is again zero.
Suppose now that $\beta>0$ and $\Delta_n^{(0)}(a)\ne\varnothing$.
Then
\[
 \#\Delta_n^{(0)}(a)
 =\#\Delta_d\!\left(\cfrac\beta t a'\right)
 =P_d\!\left(\cfrac\beta t a'\right).
\]
Choose $(x_1,\cdots,x_d,R)\in\Delta_n^{(0)}(a)$. For every $i$,
\[
 0<\cfrac{t}{\beta}\cfrac{x_i}{a_i}
 \le\sum_{j=1}^d\cfrac{t}{\beta}\cfrac{x_j}{a_j}\le1,
\]
and hence
\[
 \cfrac{\beta a_i}{t}
 \ge\cfrac{\beta a_i}{t x_i}\ge1.
\]
Theorem~\ref{thm:rough-gly} therefore implies
\[
\begin{aligned}
 d!\,\#\Delta_n^{(0)}(a)
   &=d!P_d\!\left(\cfrac\beta t a'\right)\\
   &\le\prod_{i=1}^d\left(\cfrac\beta t a_i-1\right)\\
   &=\prod_{i=1}^d
        \left(\cfrac{\beta(t+\delta_i)}t-1\right)\\
   &\le\left(\cfrac\beta t\right)^d\prod_{i=1}^d\delta_i.
\end{aligned}
\]
The final bound also applies when $\beta=0$, and when $\beta>0$ but
the slice is empty, since its right-hand side is nonnegative.

Adding the slice inequalities gives
\[
 n!P_n(a)=n!\,\#\Delta_n(a)\le U,
\]
where
\[
\begin{aligned}
 U&:=n\sum_{l=1}^{R-1}
     \left[\prod_{i=1}^d
       \left(\cfrac{\beta+l}{t}a_i-1\right)-p_d(l)\right]
     +n\left(\cfrac\beta t\right)^d\prod_{i=1}^d\delta_i\\
  &=n\sum_{l=1}^{R-1}
     \left[\prod_{i=1}^d
       \left(\beta+l-1+\cfrac{\beta+l}{t}\delta_i\right)-p_d(l)\right]
     +n\left(\cfrac\beta t\right)^d\prod_{i=1}^d\delta_i\\
  &=n\sum_{l=1}^{R-1}
     \left[\prod_{i=1}^d
       \left(s_l-1+\cfrac{s_l}{t}\delta_i\right)-p_d(l)\right]
     +n\left(\cfrac\beta t\right)^d\prod_{i=1}^d\delta_i,
\end{aligned}
\]
and $s_l:=\beta+l$. We next prove that
\[
 U\le\prod_{i=1}^n(a_i-1)-p_n(R).
\]

Let
\[
 e_j(\delta):=
 \sum_{1\le i_1<\cdots<i_j\le d}
 \delta_{i_1}\cdots\delta_{i_j}
\]
be the $j$-th elementary symmetric polynomial in
$\delta_1,\cdots,\delta_d$. Since
\[
 (l+1)_n-(l)_n=n(l)_{n-1},
\]
telescoping gives
\[
\begin{aligned}
 n\sum_{l=1}^{R-1}(l)_d
   &=\sum_{l=1}^{R-1}\bigl[(l+1)_n-(l)_n\bigr]\\
   &=(R)_n-(1)_n=(R)_n.
\end{aligned}
\]
We expand the difference in the form
\[
 \prod_{i=1}^n(a_i-1)-U-p_n(R)
 =E+\sum_{j=1}^dD_j e_j(\delta).
\]
With the convention $e_0(\delta)=1$, the product expansion is
\begin{equation}\label{eq:general-product-expansion}
 \prod_{i=1}^d(A+B\delta_i)
 =\sum_{j=0}^d A^{d-j}B^j e_j(\delta).
\end{equation}
Thus
\[
\begin{aligned}
 \prod_{i=1}^n(a_i-1)
   &=(t-1)\prod_{i=1}^d(t-1+\delta_i)\\
   &=(t-1)\sum_{j=0}^d(t-1)^{d-j}e_j(\delta)\\
   &=\sum_{j=0}^d(t-1)^{n-j}e_j(\delta),
\end{aligned}
\]
and
\begin{equation}\label{eq:slice-product-expansion}
 \prod_{i=1}^d\left(s_l-1+\cfrac{s_l}{t}\delta_i\right)
 =\sum_{j=0}^d(s_l-1)^{d-j}
   \left(\cfrac{s_l}{t}\right)^j e_j(\delta).
\end{equation}
The factor \((s_l-1)^{d-j}\) arises from the product expansion;
when \(j=d\), it represents an empty product and equals \(1\),
even if \(s_l=1\).
Consequently,
\[
\begin{aligned}
 U={}&n\sum_{l=1}^{R-1}\sum_{j=0}^d
 (s_l-1)^{d-j}\left(\cfrac{s_l}{t}\right)^j e_j(\delta)\\
 &-n\sum_{l=1}^{R-1}p_d(l)
   +n\left(\cfrac\beta t\right)^d e_d(\delta).
\end{aligned}
\]
The constant term of the difference is
\[
\begin{aligned}
 E={}&(t-1)^n
      -\left[n\sum_{l=1}^{R-1}(s_l-1)^d
              -n\sum_{l=1}^{R-1}p_d(l)\right]-p_n(R)\\
  ={}&(t-1)^n-n\sum_{l=1}^{R-1}(s_l-1)^d+n\sum_{l=1}^{R-1}\bigl[(l-1)^d-(l)_d\bigr]
       -(R-1)^n+(R)_n\\
  ={}&(t-1)^n-(R-1)^n-n\sum_{l=1}^{R-1}\bigl[(l+\beta-1)^d-(l-1)^d\bigr]
       -n\sum_{l=1}^{R-1}(l)_d+(R)_n\\
  ={}&(t-1)^n-(R-1)^n
       -n\sum_{l=1}^{R-1}\bigl[(l+\beta-1)^d-(l-1)^d\bigr].
\end{aligned}
\]
For $1\le j\le d$, the coefficient of $e_j(\delta)$ is
\[
\begin{aligned}
 D_j={}&(t-1)^{n-j}
 -n\sum_{l=1}^{R-1}(s_l-1)^{d-j}\left(\cfrac{s_l}{t}\right)^j-\mathbf{1}_{\{j=d\}}\,n\left(\cfrac\beta t\right)^d,
\end{aligned}
\]
where $\mathbf{1}_{\{j=d\}}$ is the indicator of the condition $j=d$.

First, we claim that $D_j>0$ for every $1\le j\le d$.
For $1\le l\le R-1$, we have $s_l=l+\beta<R+\beta=t$, and therefore
\[
 s_l-1<\cfrac{(t-1)s_l}{t}.
\]
It follows that
\[
\begin{aligned}
 D_j\ge{}&(t-1)^{n-j}
     -n\sum_{l=1}^{R-1}(t-1)^{d-j}\left(\cfrac{s_l}{t}\right)^d
     -n\left(\cfrac\beta t\right)^d\\
 \ge{}&(t-1)^{d-j}
 \left\{(t-1)-n\left[
     \sum_{l=1}^{R-1}\left(\cfrac{s_l}{t}\right)^d
     +\left(\cfrac\beta t\right)^d\right]\right\}.
\end{aligned}
\]
The second inequality uses $t\ge2$, hence $(t-1)^{d-j}\ge1$.
Applying Lemma~\ref{lem:power-sum} with $h=d=n-1$ gives
\[
 n\sum_{k\ge1}\bigl((t-k)_+\bigr)^d<(t-1)t^d.
\]
Since $t=R+\beta$ with $R\in\mathbb Z$ and $0\le\beta<1$, we have
\[
 \sum_{k\ge1}\bigl((t-k)_+\bigr)^d
   =\sum_{k=1}^R(R+\beta-k)^d=\sum_{l=0}^{R-1}(\beta+l)^d
    =\sum_{l=0}^{R-1}s_l^d.
\]
Thus
\[
 n\left(\beta^d+\sum_{l=1}^{R-1}s_l^d\right)<(t-1)t^d,
\]
or, equivalently,
\[
 n\left[\sum_{l=1}^{R-1}\left(\cfrac{s_l}{t}\right)^d
       +\left(\cfrac\beta t\right)^d\right]<t-1.
\]
This proves that every $D_j$ is positive.

Second, we claim that $E\ge\beta^n\ge0$. With $R$ fixed, regard $E$
as a function of $\beta\in[0,1)$:
\[
\begin{aligned}
 E(\beta)={}&(R+\beta-1)^n-(R-1)^n-n\sum_{l=1}^{R-1}
   \bigl[(l+\beta-1)^d-(l-1)^d\bigr].
\end{aligned}
\]
Differentiation yields
\[
\begin{aligned}
 \cfrac{dE}{d\beta}
   &=n\left[(R+\beta-1)^{n-1}
       -(n-1)\sum_{l=1}^{R-1}(l+\beta-1)^{n-2}\right]\\
   &=n\left[(R+\beta-1)^{n-1}
       -(n-1)\sum_{j=0}^{R-2}(j+\beta)^{n-2}\right].
\end{aligned}
\]
For $n\ge4$, the function $x\mapsto(n-1)x^{n-2}$ is increasing on
$[0,\infty)$. Hence
\[
\begin{aligned}
 (n-1)\sum_{j=0}^{R-2}(j+\beta)^{n-2}
   &\le\int_\beta^{R-1+\beta}(n-1)x^{n-2}\,dx\\
   &=(R-1+\beta)^{n-1}-\beta^{n-1}.
\end{aligned}
\]
It follows that
\[
 \cfrac{dE}{d\beta}\ge n\beta^{n-1}.
\]
Since $E(0)=0$, integration from $0$ to $\beta$ gives
\[
 E(\beta)\ge\beta^n\ge0.
\]

We have therefore proved
\[
 \prod_{i=1}^n(a_i-1)-p_n(R)-U\ge0,
\]
and consequently
\[
 n!P_n(a)\le U\le\prod_{i=1}^n(a_i-1)-p_n(R).
\]
Because
\[
 a_i\ge a_n\ge\lfloor a_n\rfloor=R\ge r,
\]
Lemma~\ref{lem:pn-monotone} gives $p_n(R)\ge p_n(r)$. Thus
\[
 n!P_n(a)\le\prod_{i=1}^n(a_i-1)-p_n(r).
\]

We now establish the equality condition. If $a_1=\cdots=a_n=r$, then
\[
 \Delta_n(a)=\left\{
 (x_1,\cdots,x_n)\in\mathbb Z_{>0}^n:
 x_1+\cdots+x_n\le r\right\}.
\]
Hence $P_n(a)=\dbinom rn$, and equality does hold.
Conversely, suppose that
\[
 n!P_n(a)=\prod_{i=1}^n(a_i-1)-p_n(r).
\]
All inequalities in the preceding chain must then be equalities, so
\[
 n!P_n(a)=U
   =\prod_{i=1}^n(a_i-1)-p_n(R)=\prod_{i=1}^n(a_i-1)-p_n(r).
\]
Lemma~\ref{lem:pn-monotone} implies $R=r$, and
\[
 \prod_{i=1}^n(a_i-1)-p_n(R)-U
 =E+\sum_{j=1}^{n-1}D_j e_j(\delta)=0.
\]
All terms on the right are nonnegative. Since $E\ge\beta^n$, we obtain
$\beta=0$ and $a_n=t=R$. Moreover,
\[
 D_1>0,\qquad e_1(\delta)=\delta_1+\cdots+\delta_{n-1},
 \qquad \delta_i\ge0,
\]
so $\delta_1=\cdots=\delta_{n-1}=0$. Therefore
$a_1=\cdots=a_n=R=r$, completing the induction.
\end{proof}

\begin{corollary}\label{cor:integer-minimum}
Let $n\ge3$. If the smallest inverse weight $a_n$ of a weighted homogeneous
isolated hypersurface singularity is an integer, then the Yau Geometric
Conjecture holds, including its equality characterization by right
equivalence to a homogeneous polynomial.
\end{corollary}

\begin{proof}
We have $a_1\ge\cdots\ge a_n>1$. If $a_n\in\mathbb Z$,
Corollary~\ref{cor:axis} gives
\[
 a_n=\lceil a_n\rceil=v\ge2.
\]
Theorem~\ref{thm:integer-reference} therefore yields
\[
 n!P_n(a)\le\prod_{i=1}^n(a_i-1)-p_n(v).
\]
Using $p_g=P_n(a)$ and $\mu=\prod_{i=1}^n(a_i-1)$, as in
\eqref{eq:pg-lattice} and \eqref{eq:milnor-product}, proves the required
inequality. Equality occurs if and only if
$a_1=\cdots=a_n=v\in\mathbb Z$.

We verify both directions of the coordinate-independent equality
statement. First, suppose that $f$ is right equivalent to a
homogeneous polynomial $g$, say
\[
 g(z)=\sum_\alpha b_\alpha z^\alpha,
 \qquad \alpha_1+\cdots+\alpha_n=m
 \quad\text{whenever }b_\alpha\ne0.
\]
Put $V_g=\{g=0\}$. Right equivalence preserves the isolated singularity condition and the invariants $\mu$, $p_g$, and $v$.
The polynomial $g$ has inverse normalized weights $(m,\ldots,m)$, and
\[
 n!p_g(V_g,0)
   =n!P_n(m,\cdots,m)=n!\dbinom mn=(m-1)^n-p_n(m).
\]
Also $v(g)=m$ and $\mu(g)=(m-1)^n$. Thus equality holds for $g$, and
the invariance of these three quantities implies equality for $f$.

Conversely, suppose that equality holds for $f$, and write
\[
 f(z)=\sum_\alpha c_\alpha z^\alpha.
\]
We then have $a_1=\cdots=a_n=v$. For every
$\alpha\in\mathbb Z_{\ge0}^n$ with $c_\alpha\ne0$, weighted
homogeneity gives
\[
 \sum_{i=1}^n\cfrac{\alpha_i}{v}
 =\cfrac1v\sum_{i=1}^n\alpha_i=1.
\]
Hence $\sum_{i=1}^n\alpha_i=v$ for every monomial occurring in $f$.
Thus $f$ itself is a homogeneous polynomial, which proves the
other direction.
\end{proof}

\section{Proof of the Yau Geometric Conjecture}\label{sec:ygc-proof}

\begin{theorem}\label{thm:high-dimensional}
Let $n\ge 7$, and let
$f:(\mathbb C^n,0)\longrightarrow(\mathbb C,0)$ be a weighted homogeneous
polynomial defining an isolated hypersurface singularity at the origin.
Retain the weights $a_1\ge a_2\ge\cdots\ge a_n>1$ introduced above, and put
\[
v=\lceil a_n\rceil,\qquad
\mu=\prod_{i=1}^n(a_i-1),\qquad p_g=P_n(a).
\]
Then
\[
n!p_g\le \mu-p_n(v).
\]
Equality holds if and only if $f$ is right equivalent to a
homogeneous polynomial.
\end{theorem}

\begin{proof}
Corollary~\ref{cor:axis} gives
$\operatorname{ord}_0f=\lceil a_n\rceil=v$. Set $u:=v-1$.
Recall the notation
\[
\Delta_n(a)=\left\{
\boldsymbol{x}=(x_1,\cdots,x_n)\in\mathbb Z_{>0}^n
:\sum_{i=1}^n\cfrac{x_i}{a_i}\le 1
\right\},\qquad P_n(a)=\#\Delta_n(a),
\]
and $p_n(v)=u^n-(v)_n$, where $(v)_n=n!\dbinom vn$.
Write
\[
f(z)=\sum_\alpha c_\alpha z^\alpha.
\]
Weighted homogeneity means that
\[
\sum_{i=1}^n\cfrac{\alpha_i}{a_i}=1
\qquad
\text{for every }\alpha\in\mathbb Z_{\ge 0}^n
\text{ with }c_\alpha\ne 0.
\]
It remains to prove
\[
\mu-n!P_n(a)\ge u^n-(v)_n
\]
and determine when equality holds.

\medskip
\noindent\textbf{(1) Pairing the small intercepts.}
Call an intercept $a_i$ \emph{small} if $a_i<v$, and let $m$ be the number
of small intercepts. If $m=0$, then $a_n=v=\lceil a_n\rceil$, so
Corollary~\ref{cor:integer-minimum} proves the assertion. If \(m=n\), then every intercept would be
smaller than $v$, and any monomial $z^\alpha$ of $f$
with nonzero coefficient would satisfy
\[
1=\sum_{i=1}^n\cfrac{\alpha_i}{a_i}
>\cfrac1v\sum_{i=1}^n\alpha_i
\ge1,
\]
a contradiction.
Henceforth suppose
that $n>m>0$.

The ordered intercepts satisfy
\[
a_1\ge\cdots\ge a_{n-m}\ge v
>a_{n-m+1}\ge\cdots\ge a_n>u=v-1\ge 1.
\]
We claim that there is an injection
\[
\sigma:\{n-m+1,\cdots,n\}\longrightarrow\{1,\cdots,n-m\}
\]
such that
\[
a_{\sigma(i)}=\cfrac{a_i}{a_i-u}>v,
\qquad n-m+1\le i\le n.
\]
We now prove this claim.

Set $q_i=1/a_i$. Then
\[
0<q_1\le\cdots\le q_{n-m}\le\cfrac1v
<q_{n-m+1}\le\cdots\le q_n<\cfrac1u\le 1.
\]
For $q\in\{q_{n-m+1},\cdots,q_n\}$, define
\[
I_q=\{j\in\{n-m+1,\cdots,n\}\mid q_j=q\}.
\]

By Lemma~\ref{lem:saito}, applied with $J=I_q$, at least one of the following holds:
\begin{enumerate}
 \item  $f$ contains a monomial with nonzero
 coefficient involving only the variables $z_j$ with $j\in I_q$.

 \item There exists a subset
 \[
  K_q\subset\{1,\cdots,n\}\setminus I_q,
  \qquad |K_q|=|I_q|,
 \]
 such that, for every $k\in K_q$, $f$ contains
 a monomial $z_km_k$ with nonzero coefficient, where $m_k$
 is a monomial involving only the variables $z_j$ with $j\in I_q$.
\end{enumerate}

The first alternative is impossible. Indeed, if \(z^\alpha\) were such a monomial, then  weighted homogeneity would give
\[
q\sum \limits_{i=1}^{n}\alpha _{i}=1.
\]
Since $q>1/v$, we obtain
\[
\deg z^\alpha=\sum \limits_{i=1}^{n}\alpha _{i}=\cfrac{1}{q} <v=\operatorname{ord}_{0}f,
\]
a contradiction.

Therefore the second alternative must hold.
 For every $k\in K_q$, weighted homogeneity gives
 \[
 q_{k}+ q\deg m_{k}=1.
 \]
 Thus,
 \[
 \deg m_{k}=\cfrac{1-q_{k}}{q}\le \cfrac{1}{q}<v,\qquad \text{and}\qquad \deg (z_k m_k)=1+\deg m_k\ge v.
 \]
Since $\deg m_k$ is an integer, we obtain
\[
\deg  m_{k}=v-1=u.
\]
Therefore,
\[
q_{k}=1-qu< 1-\cfrac{u}{v}=\cfrac{1}{v},\qquad \text{hence}\qquad k\in \{1,\cdots, n-m\}.
\]
Set \(J_q=\{j\in \{1,\,\cdots,\,n-m\}\mid q_{j}=1-qu \}\). Therefore,
\[
K_{q}\subset \{j\in \{1,\,\cdots,\,n-m\}\mid q_{j}=1-qu \}\subset \{1,\,\cdots,\,n-m\},\qquad \text{and}\qquad a_k>v.
\]
Thus we obtain
\[
|J_q|\ge|I_q|=|K_q|.
\]
We may therefore choose an injection
$\sigma_q:I_q\longrightarrow J_q$. Since $u\ge1$, the definition of $J_q$
implies that
\[
J_q\cap J_{q'}=\varnothing
\qquad
\bigl(q,q'\in\{q_{n-m+1},\cdots,q_n\},\ q\ne q'\bigr).
\]
Consequently,
\[
\sigma=\bigcup_{q\in\{q_{n-m+1},\cdots,q_n\}}\sigma_q:
\{n-m+1,\cdots,n\}\longrightarrow\{1,\cdots,n-m\}
\]
is an injection. For every $j\in\{n-m+1,\cdots,n\}$, it satisfies
\[
\cfrac1{a_{\sigma(j)}}
=q_{\sigma(j)}=1-uq_j=1-\cfrac u{a_j}.
\]
Thus
\[
a_{\sigma(j)}=\cfrac{a_j}{a_j-u},
\qquad
a_{\sigma(j)}-v=\cfrac{u(v-a_j)}{a_j-u}>0.
\]
This proves the claim. The existence of the injection also gives
\[
2m\le n.
\]

\medskip
\noindent\textbf{(2) Cases that do not require coefficient estimates.}
Step~(1) has already settled the case $m=0$. We continue to assume $m>0$,
so that $a_n$ is a small intercept.

\smallskip
\noindent\textbf{(2.1) The case $v=2$.}
Here $u=v-1=1$, and the pairing gives
\[
a_{\sigma(n)}=\cfrac{a_n}{a_n-1},
\qquad \cfrac1{a_n}+\cfrac1{a_{\sigma(n)}}=1.
\]
For every $(x_1,\cdots,x_n)\in\mathbb Z_{>0}^n$,
\begin{align*}
\sum_{i=1}^n\cfrac{x_i}{a_i}
&\ge\cfrac1{a_n}+\cfrac1{a_{\sigma(n)}}
  +\sum_{\substack{1\le i\le n\\i\notin\{n,\sigma(n)\}}}\cfrac1{a_i}=1+\sum_{\substack{1\le i\le n\\i\notin\{n,\sigma(n)\}}}\cfrac1{a_i}>1,
\end{align*}
where the strict inequality uses $n\ge3$. Hence
\[
\Delta_n(a)=\varnothing,\qquad p_g=P_n(a)=0.
\]
Moreover,
\[
\mu-p_n(v)=\prod_{i=1}^n(a_i-1)-p_n(2)
=\prod_{i=1}^n(a_i-1)-1.
\]
Theorem~\ref{thm:yau-zuo} and the Milnor--Orlik formula
\eqref{eq:milnor-product} now give
\[
\mu-p_n(v)=\mu-1\ge(v-1)^n-1\ge0.
\]
Equality holds if and only if $f$ is right equivalent to a
homogeneous polynomial.

\smallskip
\noindent\textbf{(2.2) A larger lattice-point set when $v\ge3$.}
We have
\[
a_1\ge\cdots\ge a_{n-m}\ge v
>a_{n-m+1}\ge\cdots\ge a_n>u\ge1.
\]
Define
\[
a^*=(a_1,\cdots,a_{n-m},v,\cdots,v)\in\mathbb R^n,
\qquad P^*=P_n(a^*).
\]
Every $(x_1,\cdots,x_n)\in\Delta_n(a)$ satisfies
\[
\sum_{i=1}^{n-m}\cfrac{x_i}{a_i}
+\sum_{i=n-m+1}^n\cfrac{x_i}v
\le\sum_{i=1}^n\cfrac{x_i}{a_i}\le1.
\]
It follows that
\begin{equation}\label{eq:lattice-monotonicity}
\Delta_n(a)\subset\Delta_n(a^*),
\qquad p_g=P_n(a)\le P^*.
\end{equation}

\smallskip
\noindent\textbf{(2.3) The case $v\ge3$ and $m\ge v$.}
We claim that $\Delta_n(a^*)=\varnothing$. Indeed, for every
$(x_1,\cdots,x_n)\in\mathbb Z_{>0}^n$,
\[
\sum_{i=1}^{n-m}\cfrac{x_i}{a_i}
+\sum_{i=n-m+1}^n\cfrac{x_i}v
\ge\sum_{i=1}^{n-m}\cfrac{x_i}{a_i}+\cfrac mv
>\cfrac mv\ge1.
\]
Here $2m\le n$ implies $n-m\ge m>0$, which justifies the strict
inequality. Thus $\Delta_n(a^*)=\varnothing$, and $p_g=0$.
In this case,
\[
v\le m\le\cfrac n2<n,
\]
so $(v)_n=0$ and
\[
p_n(v)=(v-1)^n-(v)_n=u^n.
\]
Theorem~\ref{thm:yau-zuo} gives
\[
\mu-p_n(v)=\mu-u^n\ge0,
\]
with equality if and only if $f$ is right equivalent to a
homogeneous polynomial.

We are therefore left with the range
\[
m>0,\qquad v\ge3,\qquad 1\le m<v,\qquad
n-m\ge\left\lceil\cfrac n2\right\rceil\ge4.
\]

\medskip
\noindent\textbf{(3) A lower bound for the Milnor number.}
Define
\[
\delta_j:=a_j-v\ge0,\qquad 1\le j\le n-m,
\]
and put $K=\sigma(\{n-m+1,\cdots,n\})$.
By Step~(1), every $j\in K$ satisfies $a_j>v$ and $\delta_j>0$.
The injectivity of $\sigma$ also gives $|K|=m$.

For $i\in\{n-m+1,\cdots,n\}$, the pairing identity
$a_{\sigma(i)}=a_i/(a_i-u)$ gives $(a_i-u)a_{\sigma(i)}=a_i$.
Consequently,
\begin{align*}
(a_i-1)(a_{\sigma(i)}-1)
&=a_i a_{\sigma(i)}-a_i-a_{\sigma(i)}+1\\
&=a_i+u a_{\sigma(i)}-a_i-a_{\sigma(i)}+1\\
&=(u-1)a_{\sigma(i)}+1\\
&=(u-1)(u+1+\delta_{\sigma(i)})+1\\
&=u^2+(u-1)\delta_{\sigma(i)}.
\end{align*}
It follows that
\begin{equation}\label{eq:paired-milnor-product}
\begin{aligned}
\mu
&=\prod_{i=n-m+1}^n(a_i-1)(a_{\sigma(i)}-1)
  \prod_{\substack{1\le j\le n-m\\j\notin K}}(a_j-1)\\
&=\prod_{i=n-m+1}^n\bigl[u^2+(u-1)\delta_{\sigma(i)}\bigr]
  \prod_{\substack{1\le j\le n-m\\j\notin K}}(\delta_j+u)\\
&=u^{2m}u^{n-2m}
  \prod_{i=n-m+1}^n\left(1+\cfrac{u-1}{u^2}\delta_{\sigma(i)}\right)
  \prod_{\substack{1\le j\le n-m\\j\notin K}}
       \left(1+\cfrac1u\delta_j\right)\\
&=u^n\prod_{\substack{1\le j\le n-m\\j\in K}}
       \left(1+\cfrac{u-1}{u^2}\delta_j\right)
  \prod_{\substack{1\le j\le n-m\\j\notin K}}
       \left(1+\cfrac1u\delta_j\right).
\end{aligned}
\end{equation}
We expand this product in the variables $\delta_j$. For
$J\subset\{1,\cdots,n-m\}$, set
\begin{equation}\label{eq:delta-product}
\delta_J=\prod_{j\in J}\delta_j,\qquad
\delta_\varnothing=1,\qquad \ell_J=|J\cap K|,
\end{equation}
and put $\rho=(u-1)/u$. Then
\begin{align*}
\mu
&=u^n\sum_{J\subset\{1,\cdots,n-m\}}
       \left(\cfrac1u\right)^{|J|-\ell_J}
       \left(\cfrac{u-1}{u^2}\right)^{\ell_J}\delta_J\\
&=u^n+\sum_{\varnothing\ne J\subset\{1,\cdots,n-m\}}
       u^{n-|J|}\rho^{\ell_J}\delta_J.
\end{align*}
Since $v\ge3$, we have $u\ge2$ and $0<\rho<1$.
Also, $\ell_J=|J\cap K|\le\min\{|J|,m\}$. Therefore
\begin{equation}\label{eq:milnor-lower-bound}
\mu\ge u^n+
\sum_{\varnothing\ne J\subset\{1,\cdots,n-m\}}
u^{n-|J|}\rho^{\min\{|J|,m\}}\delta_J.
\end{equation}

\medskip
\noindent\textbf{(4) A polynomial upper bound for $P^*$ in the variables
$\delta_j$.}
We first compute $P^*$. By definition,
\[
\Delta_n(a^*)=\left\{
x\in\mathbb Z_{>0}^n:
\sum_{j=1}^{n-m}\cfrac{x_j}{a_j}
+\cfrac{x_{n-m+1}+\cdots+x_n}{v}\le1
\right\}.
\]
For $x\in\Delta_n(a^*)$, we have
\[
m\le x_{n-m+1}+\cdots+x_n<v.
\]
Since $1\le m<v$, for $1\le r\le v-m$ we may define
\[
\Delta_n^{[r]}(a^*)=
\{x\in\Delta_n(a^*)\mid x_{n-m+1}+\cdots+x_n=v-r\}.
\]
Then
\[
\Delta_n(a^*)=\bigsqcup_{1\le r\le v-m}\Delta_n^{[r]}(a^*).
\]
If $x\in\Delta_n^{[r]}(a^*)$, then
\[
\sum_{j=1}^{n-m}\cfrac{x_j}{a_j}
\le1-\cfrac{v-r}v=\cfrac rv,
\]
or, equivalently,
\[
\sum_{j=1}^{n-m}\cfrac vr\cfrac{x_j}{a_j}\le1.
\]
Thus
\[
\Delta_n^{[r]}(a^*)=
\left\{x\in\mathbb Z_{>0}^n:
(x_1,\cdots,x_{n-m})\in
\Delta_{n-m}\left(\cfrac rv a_1,\cdots,\cfrac rv a_{n-m}\right),\,\sum_{j=n-m+1}^n x_j=v-r
\right\}.
\]
There are $\dbinom{v-r-1}{m-1}$ choices for the last $m$ positive integer
coordinates. Consequently,
\[
\bigl|\Delta_n^{[r]}(a^*)\bigr|
=\dbinom{v-r-1}{m-1}
P_{n-m}\left(\cfrac rv a_1,\cdots,\cfrac rv a_{n-m}\right),
\]
and hence
\begin{align*}
P^*
&=\sum_{r=1}^{v-m}\bigl|\Delta_n^{[r]}(a^*)\bigr|\\
&=\sum_{r=1}^{v-m}\dbinom{v-r-1}{m-1}
P_{n-m}\left(\cfrac rv a_1,\cdots,\cfrac rv a_{n-m}\right)\\
&=\sum_{r=1}^{v-m}\dbinom{v-r-1}{m-1}
P_{n-m}\left(r+\cfrac rv\delta_1,\cdots,
             r+\cfrac rv\delta_{n-m}\right).
\end{align*}

Every remaining intercept satisfies
$r+(r/v)\delta_j\ge r$, where $1\le j\le n-m$.
We also have $r\ge1$ and $n-m\ge\lceil n/2\rceil\ge4$.
For $r\ge2$ apply Theorem~\ref{thm:integer-reference}; for $r=1$ apply
Theorem~\ref{thm:rough-gly}, noting that $p_{n-m}(1)=0$. In either case,
\begin{align*}
(n-m)!P_{n-m}\left(r+\cfrac rv\delta_1,\cdots,
                         r+\cfrac rv\delta_{n-m}\right)\le
\prod_{j=1}^{n-m}\left(r-1+\cfrac rv\delta_j\right)-p_{n-m}(r).
\end{align*}
It follows that
\[
n!P^*\le
\cfrac{n!}{(n-m)!}\sum_{r=1}^{v-m}\dbinom{v-r-1}{m-1}
\left[
\prod_{j=1}^{n-m}\left(r-1+\cfrac rv\delta_j\right)-p_{n-m}(r)
\right].
\]
We now expand the product, retaining the notation
$\delta_J$ for $J\subset\{1,\cdots,n-m\}$ from Step~(3):
\begin{equation}\label{eq:pstar-product-expansion}
\prod_{j=1}^{n-m}\left(r-1+\cfrac rv\delta_j\right)
=\sum_{J\subset\{1,\cdots,n-m\}}
(r-1)^{n-m-|J|}\left(\cfrac rv\right)^{|J|}\delta_J.
\end{equation}
The factor \((r-1)^{n-m-|J|}\) arises from the product expansion;
when \(|J|=n-m\), it represents an empty product and equals \(1\),
even if \(r=1\). Therefore the upper bound can be expanded as follows:
\begin{align*}
&\cfrac{n!}{(n-m)!}\sum_{r=1}^{v-m}\dbinom{v-r-1}{m-1}
\left[
\prod_{j=1}^{n-m}\left(r-1+\cfrac rv\delta_j\right)-p_{n-m}(r)
\right]\\
=&{}\cfrac{n!}{(n-m)!}\sum_{r=1}^{v-m}\dbinom{v-r-1}{m-1}
\Biggl[-p_{n-m}(r)
+\sum_{J\subset\{1,\cdots,n-m\}}
(r-1)^{n-m-|J|}\left(\cfrac rv\right)^{|J|}\delta_J\Biggr]\\
=&{}\cfrac{n!}{(n-m)!}\sum_{r=1}^{v-m}\dbinom{v-r-1}{m-1}
\bigl[(r-1)^{n-m}-p_{n-m}(r)\bigr]
+\sum_{\varnothing\ne J\subset\{1,\cdots,n-m\}}A_{|J|}\delta_J\\
=&{}\cfrac{n!}{(n-m)!}\sum_{r=1}^{v-m}\dbinom{v-r-1}{m-1}(r)_{n-m}
+\sum_{\varnothing\ne J\subset\{1,\cdots,n-m\}}A_{|J|}\delta_J,
\end{align*}
where, for every nonempty $J\subset\{1,\cdots,n-m\}$,
\begin{equation}\label{eq:A-definition}
A_{|J|}=\cfrac{n!}{(n-m)!}\sum_{r=1}^{v-m}\dbinom{v-r-1}{m-1}
(r-1)^{n-m-|J|}\left(\cfrac rv\right)^{|J|}.
\end{equation}

The constant term satisfies
\begin{align*}
\cfrac{n!}{(n-m)!}\sum_{r=1}^{v-m}\dbinom{v-r-1}{m-1}(r)_{n-m}
&=n!\sum_{r=1}^{v-m}\dbinom{v-r-1}{m-1}\dbinom r{n-m}=n!\dbinom vn=(v)_n.
\end{align*}
To verify the identity
\[
\sum_{r=1}^{v-m}\dbinom{v-r-1}{m-1}\dbinom r{n-m}=\dbinom vn,
\]
use the polynomial identities
\begin{align*}
\sum_{r=0}^{v-1}(1+x)^{v-r-1}(1+y)^r
&=\cfrac{(1+x)^v-(1+y)^v}{x-y}\\
&=\sum_{k=1}^v\dbinom vk\cfrac{x^k-y^k}{x-y}\\
&=\sum_{k=1}^v\dbinom vk\sum_{j=0}^{k-1}x^{k-1-j}y^j
\end{align*}
and
\[
\sum_{r=0}^{v-1}(1+x)^{v-r-1}(1+y)^r
=\sum_{r=0}^{v-1}\sum_{i=0}^{v-r-1}\sum_{j=0}^r
\dbinom{v-r-1}i\dbinom rj x^iy^j.
\]
Comparing the coefficients of $x^{m-1}y^{n-m}$ proves the assertion.
The terms with $r=0$ or $r>v-m$ contribute zero to this coefficient;
if $v<n$, both sides are zero as well. We have thus obtained
\begin{equation}\label{eq:pstar-upper-bound}
n!P^*\le(v)_n+
\sum_{\varnothing\ne J\subset\{1,\cdots,n-m\}}A_{|J|}\delta_J.
\end{equation}

\medskip
\noindent\textbf{(5) Reducing the inequality to the estimate
$A_{n-m}\le(v-2)^m$.}
Since $P_n(a)\le P^*$, it suffices to prove
$n!P^*\le\mu-p_n(v)$. By the preceding bounds, it is enough to show that
\begin{equation}\label{eq:combine-bounds}
(v)_n+
\sum_{\varnothing\ne J\subset\{1,\cdots,n-m\}}A_{|J|}\delta_J\le u^n-p_n(v)+
\sum_{\varnothing\ne J\subset\{1,\cdots,n-m\}}
u^{n-|J|}\rho^{\min\{|J|,m\}}\delta_J.
\end{equation}
Because $p_n(v)=u^n-(v)_n$, it suffices to prove the coefficient bounds
\begin{equation}\label{eq:coefficients-compare}
A_{|J|}\le u^{n-|J|}\rho^{\min\{|J|,m\}},
\qquad\varnothing\ne J\subset\{1,\cdots,n-m\}.
\end{equation}
Writing $k=|J|$, with $1\le k\le n-m$, and recalling that
$u=v-1\ge2$ and $\rho=(u-1)/u\in(0,1)$, we obtain that these bounds are equivalent to
\[
\cfrac{A_k}{u^{n-k}}\le\rho^{\min\{k,m\}},
\qquad 1\le k\le n-m.
\]
We first consider the monotonicity of the sequence
$\{A_k/u^{n-k}\}_{1\le k\le n-m}$.

We have
\[
\cfrac{A_k}{u^{n-k}}
=\cfrac{n!}{(n-m)!}\sum_{r=1}^{v-m}\dbinom{v-r-1}{m-1}
\cfrac{(r-1)^{n-m-k}(r/v)^k}{u^{n-k}},
\qquad 1\le k\le n-m.
\]
Put
\begin{equation}\label{eq:T-definition}
T_{r,k}=\cfrac{(r-1)^{n-m-k}(r/v)^k}{u^{n-k}},
\qquad 1\le k\le n-m,\quad1\le r\le v-m.
\end{equation}
The factor \((r-1)^{n-m-k}\) arises from the product expansion;
when \(k=n-m\), it represents an empty product and equals \(1\),
even if \(r=1\). Then $T_{1,k}=0$ for $k<n-m$, whereas
$T_{1,n-m}=v^{m-n}u^{-m}>0$. For $r>1$ and $1\le k<n-m$,
\begin{align*}
\cfrac{T_{r,k+1}}{T_{r,k}}
&=\cfrac{(r-1)^{n-m-k-1}(r/v)^{k+1}/u^{n-k-1}}
        {(r-1)^{n-m-k}(r/v)^k/u^{n-k}}\\
&=\cfrac1{r-1}\cdot\cfrac rv\cdot u\\
&=\cfrac{ur}{v(r-1)}.
\end{align*}
Since
\[
ur-v(r-1)=v-r\ge m\ge1>0,
\]
we obtain $T_{r,k+1}>T_{r,k}$ for $r>1$. Taking the $r=1$ terms into account, we see that
$T_{r,k}$ is nondecreasing in $k$ for every $r$. Hence
$A_k/u^{n-k}$ is nondecreasing in $k$.

Since $0<\rho<1$, the sequence $\rho^{\min\{k,m\}}$ is nonincreasing.
It therefore suffices to prove
\[
\cfrac{A_{n-m}}{u^m}\le\rho^{\min\{n-m,m\}}=\rho^m,
\]
where $2m\le n$ gives $\min\{n-m,m\}=m$. Equivalently, we need
\[
A_{n-m}\le u^m\left(\cfrac{u-1}u\right)^m=(v-2)^m,
\]
if we want to prove \eqref{eq:coefficients-compare}.

\medskip
\noindent\textbf{(6) The strict estimate $A_{n-m}<(v-2)^m$.}
By definition,
\[
A_{n-m}=\cfrac{n!}{(n-m)!}\sum_{r=1}^{v-m}
\dbinom{v-r-1}{m-1}\left(\cfrac rv\right)^{n-m}.
\]
For each fixed $1\le r\le v-m$, we know that
\[
\dbinom{v-r-1}{m-1}
=\#\{y\in\mathbb Z_{>0}^m\mid y_1+\cdots+y_m=v-r\}.
\]
Define
\[
F:\mathbb R^m\longrightarrow\mathbb R,
\qquad
F(y)=\left[\left(1-\cfrac{y_1+\cdots+y_m}v\right)_+\right]^{n-m}.
\]
For $y\in\mathbb Z_{>0}^m$, if $y_1+\cdots+y_m\ge v$, then $F(y)=0$.
If $m\le y_1+\cdots+y_m<v$, then
\[
F(y)=\left(1-\cfrac{y_1+\cdots+y_m}v\right)^{n-m}.
\]
Consequently,
\begin{align*}
\sum_{y\in\mathbb Z_{>0}^m}F(y)
&=\sum_{r=1}^{v-m}
\sum_{\substack{y\in\mathbb Z_{>0}^m\\y_1+\cdots+y_m=v-r}}F(y)\\
&=\sum_{r=1}^{v-m}\dbinom{v-r-1}{m-1}
\left(\cfrac rv\right)^{n-m},
\end{align*}
and therefore
\[
A_{n-m}=\cfrac{n!}{(n-m)!}\sum_{y\in\mathbb Z_{>0}^m}F(y).
\]
Only finitely many terms of this sum are nonzero.

We now bound $A_{n-m}$ by an integral. Since $n-m\ge4$, the function
$x\mapsto(x_+)^{n-m}$ is convex. Since the map
$A:\mathbb{R}^m\rightarrow \mathbb{R},\,y\mapsto1-(y_1+\cdots+y_m)/v$ satisfies
\[
A(\lambda y+(1-\lambda )z)=\lambda A(y)+(1-\lambda)A(z),\qquad 0 \le \lambda \le 1,\;y,\,z\in \mathbb{R}^{m},
\]we have
\begin{align*}
F(\lambda y+(1-\lambda)z)&=[A(\lambda y+(1-\lambda)z)_{+}]^{n-m}=[(\lambda A(y)+(1-\lambda)A(z))_{+}]^{n-m}\\ 
&\le \lambda [A(y)_{+}]^{n-m}+(1-\lambda)[A(z)_{+}]^{n-m}\\ 
&=\lambda F(y)+(1-\lambda)F(z).
\end{align*}Therefore \(F\) is convex.
For each $k=(k_1,\cdots,k_m)\in\mathbb Z_{>0}^m$, take the cube
\[
Q_k=\prod_{i=1}^m\left[k_i-\cfrac12,k_i+\cfrac12\right].
\]
It has volume $\operatorname{Vol}Q_k = 1$ and barycenter $k$, and from the convexity of \(F\) we have
\[
F(k)=F\left( \cfrac{y+(2k-y)}{2} \right)\le \cfrac{F(y)+F(2k-y)}{2}.
\]Integrating over \(Q_k\) gives
\[
F(k)=\int_{Q_{k}}F(k)\,dy\le \cfrac{1}{2}\int_{Q_{k}}F(y)\,dy+\cfrac{1}{2}\int_{Q_{k}}F(2k-y)\,dy.
\]
The change of variables $z=2k-y$ maps $Q_k$ onto itself,
and the absolute value of its Jacobian determinant is $1$.
Hence
\[
\int_{Q_k}F(2k-y)\,dy=\int_{Q_k}F(z)\,dz,
\]
and therefore
\[
F(k)\le \int_{Q_k}F(y)\,dy.
\]
The cubes \(\{Q_k\}_{k\in \mathbb{Z}_{>0}^m}\) cover $[1/2,\infty)^m$ and overlap only on boundaries of
measure zero. Summing yields
\begin{align*}
A_{n-m}
&=\cfrac{n!}{(n-m)!}\sum_{y\in\mathbb Z_{>0}^m}F(y)\\
&\le\cfrac{n!}{(n-m)!}\sum_{y\in\mathbb Z_{>0}^m}
\int_{Q_y}F(w)\,dw\\
&=\cfrac{n!}{(n-m)!}\int_{[1/2,\infty)^m}F(w)\,dw.
\end{align*}

To calculate the last integral, put $L=v-m/2>0$ and make the change of
variables $z_i=w_i-1/2$. Then
\begin{align*}
\int_{[1/2,\infty)^m}F(w)\,dw
&=\int_{\substack{z_i\ge0\\z_1+\cdots+z_m\le L}}
\left(\cfrac{L-z_1-\cdots-z_m}v\right)^{n-m}\,dz\\
&=v^{m-n}\int_{\substack{z_i\ge0\\z_1+\cdots+z_m\le L}}
(L-z_1-\cdots-z_m)^{n-m}\,dz\\
&=v^{m-n}L^n\cfrac{(n-m)!}{n!}\\
&=v^m\left(1-\cfrac m{2v}\right)^n\cfrac{(n-m)!}{n!}.
\end{align*}
Thus
\[
A_{n-m}\le v^m\left(1-\cfrac m{2v}\right)^n.
\]
Bernoulli's inequality gives
\[
1-\cfrac m{2v}\le\left(1-\cfrac1{2v}\right)^m.
\]
Since $1\le m<v$ and $n\ge7$, it follows that
\[
A_{n-m}^{1/m}
\le v\left(1-\cfrac1{2v}\right)^n
\le v\left(1-\cfrac1{2v}\right)^7.
\]

It remains to prove that
\[
v\left(1-\cfrac1{2v}\right)^7<v-2.
\]
The function $x\mapsto x^7$ is convex for $x>0$. Since $v\ge3$ and
\[
1-\cfrac1{2v}
=\left(1-\cfrac3v\right)\cdot1+\cfrac3v\cdot\cfrac56,
\]
Jensen's inequality gives
\begin{align*}
v\left(1-\cfrac1{2v}\right)^7
&\le v\left[\left(1-\cfrac3v\right)1^7
             +\cfrac3v\left(\cfrac56\right)^7\right]\\
&=v-3+3\left(\cfrac56\right)^7\\
&<v-2.
\end{align*}
The last inequality follows from
\[
3\left(\cfrac56\right)^7=\cfrac{234375}{279936}<1.
\]
We conclude that
\[
A_{n-m}
\le\left[v\left(1-\cfrac1{2v}\right)^7\right]^m
<(v-2)^m.
\]
This proves \eqref{eq:coefficients-compare}
and hence \eqref{eq:combine-bounds}.

\medskip
\noindent\textbf{(7) Completion of the proof.}
Combining the preceding estimates in the remaining range
$v\ge3$, $1\le m<v$, and $n-m\ge4$, we obtain
\begin{align*}
n!P_n(a)
&\le n!P^*\\
&\le(v)_n+
\sum_{\varnothing\ne J\subset\{1,\cdots,n-m\}}A_{|J|}\delta_J\\
&\le u^n-p_n(v)+
\sum_{\varnothing\ne J\subset\{1,\cdots,n-m\}}
u^{n-|J|}\rho^{\min\{|J|,m\}}\delta_J\\
&\le\mu-p_n(v).
\end{align*}
Here the first, second, and fourth inequalities follow from
\eqref{eq:lattice-monotonicity},
\eqref{eq:pstar-upper-bound}, and
\eqref{eq:milnor-lower-bound}, respectively.
The third inequality follows from \eqref{eq:combine-bounds}.
This proves the inequality.

We now prove the equality characterization in both directions.

\smallskip
\noindent$(\Leftarrow)$ Suppose that $f$ is right equivalent to \(g\), where $g$ is
a homogeneous polynomial. Put $V_f=\{f=0\}$ and $V_g=\{g=0\}$.
By the invariance of $\mu$, $p_g$, and $v$ under right equivalence,
the lattice-point formula~\eqref{eq:pg-lattice}, and the
Milnor--Orlik formula~\eqref{eq:milnor-product},
\begin{align*}
n!p_g(V_f,0)
&=n!p_g(V_g,0)\\
&=\mu(g)-p_n(v(g))\\
&=\mu(f)-p_n(v(f)).
\end{align*}
Thus equality holds.

\smallskip
\noindent$(\Rightarrow)$ If $m=0$, the equality characterization was
proved in Step~(1). If $m>0$ and $v=2$, it was proved in Step~(2.1).
If $m\ge v$, it was proved in Step~(2.3).

In the remaining case $v\ge3$ and $1\le m<v$, Step~(6) gives
$A_{n-m}<(v-2)^m$, and therefore
\[
\cfrac{A_k}{u^{n-k}}
\le\cfrac{A_{n-m}}{u^m}
<\rho^m\le\rho^{\min\{k,m\}},
\qquad1\le k\le n-m.
\]
Since $m>0$, the set $K$ is nonempty. Choose $j\in K$; then
$\delta_j=a_j-v>0$. Consequently,
\[
A_1\delta_{\{j\}}<u^{n-1}\rho\delta_{\{j\}}.
\]
All the other monomials $\delta_J$ are nonnegative, so
\begin{align*}
&(v)_n+
\sum_{\varnothing\ne J\subset\{1,\cdots,n-m\}}A_{|J|}\delta_J<u^n-p_n(v)+
\sum_{\varnothing\ne J\subset\{1,\cdots,n-m\}}
u^{n-|J|}\rho^{\min\{|J|,m\}}\delta_J.
\end{align*}
Hence in this case from 
\eqref{eq:lattice-monotonicity},
\eqref{eq:milnor-lower-bound}, and
\eqref{eq:pstar-upper-bound} we have
\begin{align*}
n!P_n(a)
&\le n!P^*\\
&\le(v)_n+
\sum_{\varnothing\ne J\subset\{1,\cdots,n-m\}}A_{|J|}\delta_J\\
&< u^n-p_n(v)+
\sum_{\varnothing\ne J\subset\{1,\cdots,n-m\}}
u^{n-|J|}\rho^{\min\{|J|,m\}}\delta_J\\
&\le\mu-p_n(v).
\end{align*}
Thus equality is impossible. This completes the proof.
\end{proof}

\begin{corollary}[Theorem~\ref{thm:main-introduction}, restated]\label{cor:ygc-all}
The Yau Geometric Conjecture holds. More precisely, let $n\ge3$, and let
$f:(\mathbb C^n,0)\longrightarrow(\mathbb C,0)$ be a weighted homogeneous
polynomial defining an isolated hypersurface singularity at the origin.
Let $\mu$, $p_g$, and $v$ denote, respectively, the Milnor number, geometric
genus, and multiplicity of $V=\{z\mid f(z)=0\}$. Then
\[
\mu-p_n(v)\ge n!p_g,
\]
where
\[
p_n(v)=(v-1)^n-v(v-1)\cdots(v-n+1)
=(v-1)^n-n!\dbinom vn,
\]
with the convention $\dbinom vn=0$ for $v<n$.
Equality holds if and only if $f$ is right equivalent to a
homogeneous polynomial.
\end{corollary}

\begin{proof}
For $n\ge7$, the inequality and equality characterization follow from
Theorem~\ref{thm:high-dimensional}, the Milnor--Orlik
formula~\eqref{eq:milnor-product}, the geometric-genus
formula~\eqref{eq:pg-lattice}, and Corollary~\ref{cor:axis}.

It remains to consider $3\le n\le6$. If \(p_g=0\), the conjecture is proved in \cite[Theorem~C]{YZ15}. If $p_g>0$, the case $n=3$ is proved in
\cite[Theorem~A]{XY93}; the case $n=4$ is proved in
\cite[Main Theorem]{LY04}; the case $n=5$ is proved in
\cite[Theorem~1.4]{CLYZ11}; and the case $n=6$ is proved in
\cite[Theorem~1.10]{LRYZ18}.
These results provide the required inequality and equality
characterization in the respective dimensions. All cases are now
covered.
\end{proof}

\end{document}